\documentclass{amsart}
\usepackage{amssymb,amstext,amsmath,amscd,amsthm,amsfonts,enumerate,graphicx,latexsym,stmaryrd,multicol}

\usepackage[usenames]{color}
\usepackage[all]{xy}
\newtheorem{thm}{Theorem}[section]
\newtheorem{lem}[thm]{Lemma}
\newtheorem{prop}[thm]{Proposition}
\newtheorem{cor}[thm]{Corollary}
\theoremstyle{definition}

\newtheorem{ques}[thm]{Question}
\newtheorem{rem}[thm]{Remark}

\newtheorem{ex}[thm]{Example}

\newtheorem{conv}[thm]{Convention}
\theoremstyle{remark}

\newtheorem{claim}{Claim}
\newtheorem{claim2}{Claim}
\newtheorem*{ac}{Acknowledgments}
\def\Ann{\mathsf{Ann}}
\def\cok{\operatorname{Cok}}
\def\Ext{\operatorname{Ext}}
\def\dual{D}
\def\dim{\operatorname{dim}}
\def\grade{\operatorname{grade}}
\def\H{\operatorname{H}}
\def\height{\operatorname{ht}}
\def\Hom{\operatorname{Hom}}
\def\image{\operatorname{Im}}
\def\ker{\operatorname{Ker}}
\def\L{\mathbf{L}}
\def\lim{\operatorname{lim}}
\def\m{\mathfrak{m}}

\def\p{\mathfrak{p}}
\def\pd{\operatorname{pd}}
\def\q{\mathfrak{q}}
\def\rank{\operatorname{rank}}
\def\RHom{\operatorname{\mathbf{R}Hom}}

\def\spec{\operatorname{Spec}}
\def\supp{\operatorname{Supp}}
\def\tr{\operatorname{tr}}
\begin{document}
\allowdisplaybreaks
\title{Trace ideals, conductors, and ideals of finite (phantom) projective dimension}
\author{Kaito Kimura}
\address{Graduate School of Mathematics, Nagoya University, Furocho, Chikusaku, Nagoya 464-8602, Japan}
\email{m21018b@math.nagoya-u.ac.jp}
\begin{abstract}
In this paper, we consider whether parameter test ideals, conductors, $F$-ideals, and trace ideals are contained in an ideal whose quotient ring has finite phantom projective dimension (for example, ideals generated by a system of parameters or ideals with finite projective dimension).
One of the main results asserts that such inclusions do not exist in quasi-Gorenstein complete local domains.
We also provide examples of Cohen--Macaulay local rings with good properties where such inclusions occur, thus answering negatively a question of Huneke-Swanson. 
\end{abstract}
\maketitle

\section{Introduction}

Throughout the present paper, all rings are assumed to be commutative and Noetherian. 
We refer to Conventions \ref{conv} and \ref{def} for the notation and terminology used in this paper.

The first motivation of this paper is the following question for analytically unramified local rings:
\begin{ques}\label{ques 1}
Is the conductor not contained in any ideal generated by a system of parameters?
\end{ques}
\noindent Ikeda \cite{I} proved that the monomial conjecture holds if and only if Question \ref{ques 1} is affirmative for every Gorenstein complete local domain, and raised the issue of which domains satisfy this question.
By now, the monomial conjecture has been resolved positively by Andr\'{e} \cite{And}, so Question \ref{ques 1} holds for every Gorenstein complete local domain.
Asgharzadeh \cite{A} showed that Question \ref{ques 1} holds for every analytically unramified quasi-Gorenstein local ring.
Huneke and Swanson posed the open problem of whether Question \ref{ques 1} holds for all analytically unramified Cohen--Macaulay local domains; see \cite[Exercise 12.5]{HS}.

Another motivation is the following question about local rings of prime characteristic:
\begin{ques}\label{ques 2}
Is the parameter test ideal not contained in any ideal generated by a system of parameters?
\end{ques}
\noindent For a reduced local ring of prime characteristic, the conductor contains the parameter test ideal, and hence the validity of Question \ref{ques 2} implies that of Question \ref{ques 1}.
Smith \cite{S2}\footnote{In \cite{S2}, it is stated that Question \ref{ques 2} holds for any Cohen--Macaulay local ring, but the proof contains a gap and a counterexample exists as given in this paper. For example, if we assume the ring is complete, Gorenstein, and reduced, the proof works; see Remark \ref{others remark}(4).} showed that Question \ref{ques 2} (and thus Question \ref{ques 1}) holds for any Gorenstein reduced complete local ring of prime characteristic.
Dey and Dutta \cite{DD} proved that the parameter test ideal of a Gorenstein complete local domain is not contained in any ideal of finite projective dimension.
In this paper, we aim to generalize these results in two different directions.
One of our main results is a unified generalization of several results, including those previously established by Asgharzadeh \cite{A}, Dey-Dutta \cite{DD}, Ikeda \cite{I}, and Smith \cite{S2}.
We also construct a counterexample to the problem posed by Huneke and Swanson.

In \cite{S1}, Smith introduced the concept of $F$-ideal, as a key ingredient of the result in \cite{S1}. Smith showed in \cite{S2} that a non-zero $F$-ideal is not contained in a parameter ideal.
Given that a parameter ideal of a Cohen--Macaulay ring has finite projective dimension, the following theorem extends Smith's result.

\begin{thm}[Theorem \ref{F-ideal and fpd}]\label{main 1}
Let $R$ be a local ring of prime characteristic $p$, $I$ a proper ideal of $R$ of finite projective dimension, and $J$ a non-zero $F$-ideal.
If $R$ is either a Cohen--Macaulay ring or an excellent equidimensional reduced ring, then $J$ is not contained in $I$.
\end{thm}

\noindent In the proof of this result, the theorem of Hochster and Yao \cite{HY} plays an important role.
Theorem \ref{main 1} (and Remark \ref{para ver of main 1}) asserts that the same statement holds under mild assumptions, even for rings that are not necessarily Cohen--Macaulay.
In this case, the gap arising from the system of parameters not forming a regular sequence is bridged by the existence of a test element.
As a corollary of Theorem \ref{main 1} (and Remark \ref{para ver of main 1}), it follows that the parameter test ideal and the conductor of a quasi-Gorenstein equidimensional reduced complete local ring are neither contained in an ideal of finite projective dimension nor a parameter ideal, thereby refining both Smith's result, and the result of Dey and Dutta.

Next, we consider the problem posed by Huneke and Swanson.
Question \ref{ques 1} has an affirmative answer for every one-dimensional analytically unramified local ring; see Remark \ref{others remark}(5).
The problem lies in cases of dimension two or higher.
We provide an example (see Example \ref{count. para. s.o.p}) of a two-dimensional analytically unramified Cohen--Macaulay local domain such that Question \ref{ques 1} (and hence Question \ref{ques 2}) does not hold.

Conductors and parameter test ideals often contain a trace ideal of some big Cohen--Macaulay module.
Dey and Dutta \cite{DD} approached Question \ref{ques 2} from this perspective.
They proved the following: when $R$ is a Cohen--Macaulay complete local ring, the trace submodule $\tr_\omega (M)$ of $\omega$ is not contained in $I\omega$ for a proper ideal $I$ of $R$ of finite projective dimension, where $M$ is a big Cohen--Macaulay $R$-module and $\omega$ is the canonical module of $R$.
In particular, the trace ideal $\tr_R (M)$ is not contained in $I$ if $R$ is Gorenstein.
Our main result below extends their result in several directions.

\begin{thm}[Corollary \ref{nonCM original}]\label{main 2}
Let $R$ be a complete local ring, $I$ a proper ideal of $R$, $M$ a big Cohen--Macaulay $R$-module, and $\omega$ a canonical module of $R$.
Suppose that there is a bounded complex $F$ of finitely generated free $R$-modules satisfying the standard conditions on rank and height such that $H_0(F)=R/I$ (for example, when $I$ has finite projective dimension or $I$ is a parameter ideal).
Then $\tr_\omega (M)\nsubseteq I\omega$.
Additionally, assume that $R$ is quasi-Gorenstein.
Then the trace ideal $\tr_R (M)$ is not contained in $I$.
In particular, the conductor of $R$ is not contained in $I$ if $R$ is a complete local domain.
\end{thm}

\noindent The proof by Dey and Dutta utilized the vanishing of Tor and Ext modules, which arise from the maximal Cohen--Macaulay property or the finiteness of injective dimension of the canonical module.
However, these properties depend on the Cohen--Macaulayness of the ring.
Our main approach to addressing this problem is a careful analysis of the derived functors defined by the dualizing complex.
In this process, we observe that the finiteness of projective dimension can be replaced by standard conditions on rank and height.

Theorem \ref{main 2} generalizes not only the result of Dey and Dutta but also several other results.
When the ring is a complete local domain, the conductor contains the trace ideal of any big Cohen--Macaulay $R^{+}$-algebra.
Hence, in this case, Theorem \ref{main 2} refines Asgharzadeh's result in two directions.
In particular, combining with the result of Ikeda, we see that it is also a generalization of the monomial conjecture.
On the other hand, it is known that the parameter test ideal is the trace ideal of certain big Cohen--Macaulay algebra when the ring is a complete local domain of prime characteristic.
The corollary below, which follows from Theorem \ref{main 2}, improves Smith's result and the result of Dey and Dutta in a different direction from Theorem \ref{main 1}.

\begin{cor}[Corollary \ref{nonCM cor}]\label{main 3}
Let $R$ be a complete local domain of prime characteristic, $I$ a proper ideal of $R$, and $\omega$ a canonical module of $R$.
Suppose that $R/I$ has finite phantom projective dimension (for example, when $I$ has finite projective dimension or $I$ is a parameter ideal).
Then $\tau (\omega)\nsubseteq I\omega$.
In particular, if $R$ is quasi-Gorenstein, the parameter test ideal of $R$ is not contained in $I$.
\end{cor}

\noindent From the above, it is concluded that a stronger assertion than Questions \ref{ques 1} and \ref{ques 2} holds when the ring is a quasi-Gorenstein complete domain, not necessarily Cohen--Macaulay.

The organization of this paper is as follows. 
In Section 2, we study parameter test ideals, conductors, and $F$-ideals.
We prove Theorem \ref{main 1} and provide counterexamples to Questions \ref{ques 1} and \ref{ques 2}.
In Section 3, we consider the derived functors defined by the dualizing complex.
We show Theorem \ref{main 2} and Corollary \ref{main 3} by applying these functors to complexes of free modules satisfying the standard conditions.

\section{Parameter test ideal and $F$-ideal}

In this section, we investigate when the parameter test ideal and the conductor are contained in an ideal generated by a system of parameters or an ideal with finite projective dimension, in the case where the ring has prime characteristic.
First, we state the notation used in this paper and the basic facts required for the proofs.

\begin{conv}\label{conv}
Let $(R,\m)$ be a $d$-dimensional local ring of prime characteristic $p$.

(1) The \textit{Frobenius endomorphism} is the ring homomorphism $F:R\to R: a\mapsto a^p$.
For any $e>0$, we denote by $F^e:R\to R$ the $e$-th iterate of the Frobenius endomorphism, that is, $F^e(a)=a^{p^e}$ for $a\in R$.
Denote by $F^e_\ast (R)$ the target of $F^e$ as a module over the source.
Let $N\subseteq M$ be $R$-modules.
The \textit{tight closure} $N^\ast_M$ of $N$ in $M$ is defined as follows: for $z\in M$, $z\in N^\ast_M$ if there is $c\in R$, not in any minimal prime of $R$, such that $c\otimes z\in \image(F^e_\ast (R)\otimes_R N \to F^e_\ast (R)\otimes_R M)$ for all $e\gg 0$.
We write simply $N^\ast$ when there is no risk of confusion.

(2) Let $I$ be an ideal of $R$ and $q>0$ an integer.
The ideal generated by $q$-th powers of the elements of $I$ is denoted by $I^{[q]}$. 
If $I=(a_1,\ldots,a_n)$ and $q=p^e$ for some $e>0$, then $I^{[q]}=(a_1^q,\ldots,a_n^q)$ holds and there exists a natural isomorphism $F^e_\ast (R) \otimes_R R/I\cong R/I^{[q]}$.
For any $z\in R$, $z\in I^\ast$ if and only if there is $c\in R$, not in any minimal prime of $R$, such that $cz^{p^e}\in I^{[p^e]}$ for all $e\gg 0$.

(3) Let $x_1,\ldots, x_d$ be a system of parameters of $R$.
Then the family of maps 
$$
R/(x_1^n,\ldots, x_d^n)\to R/(x_1^{n+1},\ldots, x_d^{n+1}): a+(x_1^n,\ldots, x_d^n) \mapsto a\prod_{i=1}^d x_i +(x_1^{n+1},\ldots, x_d^{n+1})
$$ 
induces an isomorphism $\varinjlim_{n}R/(x_1^n,\ldots, x_d^n)\cong \H_\m^d(R)$.
An element of $\H_\m^d(R)$ is denoted by the equivalence class $[z+(x_1^n,\ldots, x_d^n)]$.
The natural Frobenius action on $\H_\m^d(R)$ is defined as follows:
$$
F:\varinjlim_{n}R/(x_1^n,\ldots, x_d^n) \to \varinjlim_{n}R/(x_1^{pn},\ldots, x_d^{pn}): [z+(x_1^n,\ldots, x_d^n)] \mapsto [z^p+(x_1^{pn},\ldots, x_d^{pn})].
$$
Following Smith \cite{S1}, an $R$-submodule $N$ of $\H_\m^d(R)$ is said to be \textit{$F$-stable} if $F(N)\subseteq N$.
An ideal $J$ of $R$ is an \textit{$F$-ideal} if the submodule $(0:_{\H_\m^d(R)}J)$ of $\H_\m^d(R)$ is $F$-stable.
Since tensor product commutes with direct limit, $\eta\in\H_\m^d(R)$ belongs to $0_{\H_\m^d(R)}^\ast$ if and only if for some fixed $c\in R$ which is not in any minimal prime of $R$, $cF^e(\eta)=0$ in $\H_\m^d(R)$ for all $e\gg 0$.

(4) Following Smith \cite{S2}, an ideal $J$ of a ring is said to be a \textit{parameter ideal} if it is generated by $\height J$ elements.
The \textit{parameter test ideal} of $R$ is the ideal of all elements $c\in R$ such that $cJ^\ast\subseteq J$ for every parameter ideal $J$. 
Let $\boldsymbol{x}=x_1,\ldots, x_r$ be a sequence in $R$.
The \textit{limit closure} $(\boldsymbol{x})^{\lim}$ of $\boldsymbol{x}$ is defined by $(\boldsymbol{x})^{\lim}=\bigcup_{n>0} ( (x_1^{n+1},\ldots, x_r^{n+1}):_{R} \prod_{i=1}^r x_i^n)$.
Suppose that $r=d$ and $\boldsymbol{x}$ is a system of parameters of $R$.
Then $z\in (\boldsymbol{x})^{\lim}$ if and only if $[z+(x_1,\ldots, x_d)]=0$ in $\varinjlim_{n}R/(x_1^n,\ldots, x_d^n)\cong \H_\m^d(R)$.
If $R$ is Cohen--Macaulay, $(\boldsymbol{x})^{\lim}=(x_1,\ldots, x_d)$ holds.
Let $S$ be a reduced ring with total ring of fraction $K$ and integral closure $\overline{S}$.
The \textit{conductor} of $S$ is $(S:_K \overline{S})$, which is the largest common ideal of $S$ and $\overline{S}$.
\end{conv}

There is also a convention to call an ideal of a local ring a parameter ideal if it is generated by part of a system of parameters.
Note that this coincides with the definition in Convention \ref{conv} when the ring is catenary and equidimensional, but in general, they may differ.
The following is a key fact for considering Questions \ref{ques 1} and \ref{ques 2} in the context of reduced local rings of prime characteristic.

\begin{rem}\label{conduc para test}
Let $R$ be a reduced local ring of prime characteristic and $\overline{R}$ the integral closure of $R$.
Then the parameter test ideal $J$ of $R$ is contained in the conductor $\mathfrak{C}$ of $R$.
In fact, let $z\in J$ and $x/y\in \overline{R}$.
Since $x/y$ is integral over $R$, $x$ belongs to the integral closure of $(y)$.
Then $x\in (y)^\ast$ by \cite[Theorem 5.4]{HH2}.
As $z\in J$ and $(y)$ is a parameter ideal, we have $zx\in(y)$ and thus $z(x/y)\in R$.
We conclude that $z\in (R:_K \overline{R})=\mathfrak{C}$.
When $R$ is one-dimensional and excellent, the converse holds.
For any ideal $I$ of $R$, $I^\ast\subseteq I^\ast \overline{R}\subseteq (I \overline{R})^\ast=I \overline{R}$ holds as $\overline{R}$ is regular.
We obtain $\mathfrak{C}I^\ast \subseteq \mathfrak{C} I \overline{R}\subseteq I$.
\end{rem}

In \cite[Lemma 4.6]{S1}, a necessary and sufficient condition for an ideal of a Cohen--Macaulay local ring to be an $F$-ideal is provided in terms of a system of parameters.
To prove the main result of this section, we give its general case.

\begin{lem}\label{nonCM F-ideal equiv.}
Let $(R,\m)$ be a $d$-dimensional local ring of prime characteristic $p$.
For any ideal $J$ of $R$, the following are equivalent:
\begin{enumerate}[\rm(1)]
\item $J$ is an $F$-ideal; 
\item For any system of parameters $x_1,\ldots, x_d$ of $R$ and $z\in R$, if $Jz\subseteq (x_1,\ldots, x_d)^{\lim}$, then $Jz^p\subseteq (x_1^p,\ldots, x_d^p)^{\lim}$;
\item For some fixed system of parameters $x_1,\ldots, x_d$ of $R$, if there exists $z\in R$ and $t>0$ with the property that $Jz\subseteq (x_1^t,\ldots, x_d^t)^{\lim}$, then $Jz^p\subseteq (x_1^{pt},\ldots, x_d^{pt})^{\lim}$.
\end{enumerate}
\end{lem}

\begin{proof}
The implication (2)$\Rightarrow$(3) is obvious.
Suppose that $J$ is an $F$-ideal.
If $Jz\subseteq (x_1,\ldots, x_d)^{\lim}$ for $z\in R$, then $J [z+(x_1,\ldots, x_d)]=0$ in $\varinjlim_{n} R/(x_1^n,\ldots, x_d^n)\cong \H_\m^d(R)$.
Since $J$ is an $F$-ideal, $[z^p+(x_1^p,\ldots, x_d^p)]\in \varinjlim_{n} R/(x_1^{pn},\ldots, x_d^{pn})\cong\H_\m^d(R)$ is annihilated by $J$ as well.
By this, we have $Jz^p\subseteq (x_1^p,\ldots, x_d^p)^{\lim}$.
Thus (1)$\Rightarrow$(2) holds.
To prove (3)$\Rightarrow$(1), we assume that the condition (3) holds.
Let $[z+(x_1^t,\ldots, x_d^t)]$ be an element of $\varinjlim_{n} R/(x_1^n,\ldots, x_d^n)\cong\H_\m^d(R)$ annihilated by $J$, where $z\in R$.
Then $Jz\subseteq (x_1^t,\ldots, x_d^t)^{\lim}$.
We obtain $Jz^p\subseteq (x_1^{pt},\ldots, x_d^{pt})^{\lim}$ by assumption.
This yields that $[z^p+(x_1^{pt},\ldots, x_d^{pt})]\in \varinjlim_{n} R/(x_1^{pn},\ldots, x_d^{pn})\cong\H_\m^d(R)$ is annihilated by $J$, which means that $J$ is an $F$-ideal.
\end{proof}

Smith \cite{S2} proved that a non-zero $F$-ideal of Cohen--Macaulay local ring is not contained in an ideal generated by a system of parameters.
The following theorem refines this result.

\begin{thm}\label{F-ideal and fpd}
Let $R$ be a local ring of prime characteristic $p$, $I$ a proper ideal of $R$ of finite projective dimension, and $J$ a non-zero $F$-ideal.
If $R$ is either a Cohen--Macaulay ring or an excellent equidimensional reduced ring, then $J\nsubseteq I$.
\end{thm}

\begin{proof}
Suppose $J\subseteq I$.
Put $s=\pd_R (R/I)$ and $d=\dim R$.
By \cite[Theorem 2.3]{HY}, there exist a proper $R$-regular sequence $x_1,\ldots,x_s$, non-negative integers $n_0, \ldots, n_s$, and a short exact sequence
$$
0\to R/I \xrightarrow{\varphi} \oplus_{i=0}^s (R/(x_1,\ldots, x_i))^{n_i}\to N\to 0.
$$
Note that $N$ has finite projective dimension.
We take $x_{s+1},\ldots, x_d$ such that $x_1,\ldots,x_s, x_{s+1},\ldots, x_d$ is a system of parameters of $R$.
Let $\varphi$ be given by the matrix $(r_{ij})_{0\le i\le s, 1\le j\le n_i}$, where $r_{ij}\in R$.
For each $e>0$, applying $F^e_\ast (R)\otimes_R (-)$ to the minimal free resolution of $N$, we have $\operatorname{Tor}_1^R(F^e_\ast (R), N)=0$ by \cite[Th\'{e}or\`{e}me 1.7]{PS} and obtain the exact sequence
\begin{equation}
0\to R/I^{[p^e]} \xrightarrow{F^e_\ast (R)\otimes \varphi} \oplus_{i=0}^s (R/(x_1^{p^e},\ldots, x_i^{p^e}))^{n_i}\to F^e_\ast (R)\otimes_R N\to 0 \tag*{(\ref{F-ideal and fpd}.1)}.
\end{equation}
Then $F^e_\ast (R)\otimes \varphi$ is given by the matrix $(r_{ij}^{p^e})_{0\le i\le s, 1\le j\le n_i}$.
Since $J$ is contained in $I$, $J$ annihilates $\image \varphi$.
This means that $Jr_{ij} \subseteq (x_1,\ldots, x_i)$ for all $0\le i\le s$ and $1\le j\le n_i$, and hence $Jr_{ij} \subseteq (x_1,\ldots, x_i, x_{i+1}^m, \ldots,x_d^m)$ for any $m>0$.
Lemma \ref{nonCM F-ideal equiv.} says that $Jr_{ij}^{p^e} \subseteq (x_1^{p^e},\ldots, x_i^{p^e}, x_{i+1}^{m p^e}, \ldots,x_d^{mp^e})^{\lim}$ for all $e>0$.

If $R$ is Cohen--Macaulay, then $(x_1^{p^e},\ldots, x_i^{p^e}, x_{i+1}^{m p^e}, \ldots,x_d^{mp^e})^{\lim}=(x_1^{p^e},\ldots, x_i^{p^e}, x_{i+1}^{m p^e}, \ldots,x_d^{mp^e})$ for every $i,m,e$.
We have $Jr_{ij}^{p^e} \subseteq \bigcap_{m>0} (x_1^{p^e},\ldots, x_i^{p^e}, x_{i+1}^{m p^e}, \ldots,x_d^{mp^e})=(x_1^{p^e},\ldots, x_i^{p^e})$ for all $0\le i\le s$, $1\le j\le n_i$, and $e>0$.
Thus for any $e>0$, $J$ annihilates $\image (F^e_\ast (R)\otimes \varphi)$, which means that $J\subseteq I^{[p^e]}$ by {(\ref{F-ideal and fpd}.1)}.
We get $J\subseteq \bigcap_{e>0} I^{[p^e]} =0$, a contradiction.

If $R$ is excellent, equidimensional, and reduced, it is a homomorphic image of a Cohen--Macaulay local ring by \cite[Corollary 1.2]{Kaw}.
Then  for each $i,j,m,e$, $Jr_{ij}^{p^e} \subseteq(x_1^{p^e},\ldots, x_i^{p^e}, x_{i+1}^{m p^e}, \ldots,x_d^{mp^e})^{\lim}\subseteq (x_1^{p^e},\ldots, x_i^{p^e}, x_{i+1}^{m p^e}, \ldots,x_d^{mp^e})^\ast$ by \cite[Theorem 2.3]{Hu}.
It follows from \cite[Theorem 6.1(a)]{HH} that there exists $c\in R$, not in any minimal prime ideal of $R$, such that $c Jr_{ij}^{p^e} \subseteq (x_1^{p^e},\ldots, x_i^{p^e}, x_{i+1}^{m p^e}, \ldots,x_d^{mp^e})$.
A similar argument to that in the previous paragraph shows that $c J=0$.
Since $R$ is reduced, $c$ is a non-zero-divisor of $R$.
Hence $J=0$, a contradiction.
\end{proof}

\begin{rem}\label{para ver of main 1}
If $R$ is a Cohen--Macaulay local ring of prime characteristic $p$, any parameter ideal has finite projective dimension.
Therefore, the former part of Theorem \ref{F-ideal and fpd} derives the latter part of \cite[Proposition 6.1]{S2}.
On the other hand, we see that any non-zero $F$-ideal is not contained in a parameter ideal when $R$ is excellent, equidimensional, and reduced.
In fact, if an $F$-ideal $J$ is contained in a parameter ideal $I$, there exists a system of parameters $x_1,\ldots, x_d$ such that $J\subseteq (x_1,\ldots, x_d)$ by extending a minimal system of generator of $I$.
By Lemma \ref{nonCM F-ideal equiv.}, we have $J\subseteq (x_1^{p^e},\ldots, x_d^{p^e})^{\lim}$ for any $e>0$.
An analogous argument to the last paragraph of the proof of Theorem \ref{F-ideal and fpd} shows that there exists a non-zero-divisor $c\in R$ such that $c J\subseteq (x_1^{p^e},\ldots, x_d^{p^e})$ for all $e>0$.
We obtain $c J\subseteq \bigcap_{e>0} (x_1^{p^e},\ldots, x_d^{p^e})=0$ and thus $J=0$.
\end{rem}

Smith \cite{S2} also showed that the parameter test ideal is an $F$-ideal under certain assumptions; see Remark \ref{others remark}(4).
Below is a direct corollary of Theorem \ref{F-ideal and fpd} and Remark \ref{para ver of main 1}.

\begin{cor}\label{Check}
Let $R$ be a quasi-Gorenstein equidimensional reduced complete local ring of prime characteristic, and $I$ either a proper ideal of finite projective dimension or a parameter ideal of $R$.
Then the parameter test ideal and the conductor of $R$ are not contained in $I$.
\end{cor}

\begin{proof}
The assertion follows from Theorem \ref{F-ideal and fpd}, Remarks \ref{conduc para test}, \ref{para ver of main 1}, and \ref{others remark}(4).
\end{proof}


Remark \ref{para ver of main 1} shows that Cohen--Macaulayness of the ring is not necessarily required to determine whether an $F$-ideal is contained in an ideal generated by a system of parameters.
Such observations about Cohen--Macaulayness can also be inferred from the results in Section 3.
Conversely, these observation suggest that even for Cohen--Macaulay rings with sufficiently good properties, there may be counterexamples to Questions \ref{ques 1} and \ref{ques 2}.
Indeed, two counterexamples are provided in the remainder of this section.
We give a simple remark.

\begin{rem}\label{domain rem}
Let $K$ be a field, and let $a,b$ be positive integers that are coprime.
Assume that $S$ is a polynomial ring $K[x_1,\ldots,x_n]$ (resp. a formal power series ring $K\llbracket x_1,\ldots,x_n\rrbracket$) over $K$ and that $R$ is $K[x_1,\ldots,x_n, y]/(y^a-\alpha)$ (resp. $K\llbracket x_1,\ldots,x_n, y\rrbracket /(y^a-\alpha)$), where $\alpha\in S$ such that $\alpha \in z^b S\setminus z^{b+1} S$ for some prime element $z\in S$.
Then $R=S[y]/(y^a-\alpha)$.
Let $T$ be the fraction field of $S$.
Since $y^a-\alpha$ is a monic polynomial, $R$ is flat over $S$.
Applying $(-)\otimes_S R$ to the natural injection $S\to T$, we see that $R\to T[y]/(y^a-\alpha)$ is injective.
By \cite[Chapter 8, Theorem 1.6]{Kar}, $y^a-\alpha$ is irreducible in $T[y]$.
In fact, if $\alpha=(f/g)^p$ in $T$ for some $f,g\in S$ and some prime $p$ dividing $a$, then $\alpha g^p=f^p$ in $S$.
Since $S$ is a unique factorization domain and $\alpha \in z^b S\setminus z^{b+1} S$ for some prime element $z$, this contradicts the fact that $b$ is not divisible by $p$.
Similarly, when $a$ is divisible by 4, we see that $\alpha\ne -4(f/g)^4$ for any $f,g\in S$.
Hence $T[y]/(y^a-\alpha)$ is an integral domain, and so is $R$.
\end{rem}

The following two examples were proposed by Linquan Ma, and proofs were provided by him and the author.
The author is indebted to him for his offer to include them in this paper.
The following is an example of a Cohen--Macaulay reduced complete local ring where the parameter test ideal is not an $F$-ideal, providing a counterexample to \cite[Proposition 4.5]{S2}.

\begin{ex}\label{count. para. F-ideal CM}
Let $K$ be a field of characteristic $p>0$, $R=K \llbracket x, y, z, v, w \rrbracket/(xy, xz, yz, yv, zw, x^2-vw)$ a quotient of a formal power series ring over $K$, and $\m=(x, y, z, v, w)R$.

\begin{claim}
$(xy, xz, yz, yv, zw, x^2-vw)=(x,y,w)\cap(x,z,v)\cap(y,z,x^2-vw)$ is a primary decomposition of $(xy, xz, yz, yv, zw, x^2-vw)$ in $K \llbracket x, y, z, v, w \rrbracket$.
Hence $R$ is reduced, equidimensional, and of dimension 2.
\end{claim}

The inclusion $(xy, xz, yz, yv, zw, x^2-vw)\subseteq (x,y,w)\cap(x,z,v)\cap(y,z,x^2-vw)$ is obvious.
To prove the converse, consider $Ay+Bz+C(x^2-vw)\in (x,y,w)\cap(x,z,v)$, where $A,B,C\in K \llbracket x, y, z, v, w \rrbracket$.
Since $Ay\in (x,y,w)\cap(x,z,v)+(z, x^2-vw)\subseteq (x,z,v)$ and $Bz\in (x,y,w)\cap(x,z,w)+(y, x^2-vw)\subseteq (x,y,w)$, we obtain $A\in (x,z,v)$ and $B\in (x,y,w)$.
Therefore, $Ay+Bz+C(x^2-vw)\in (xy, xz, yz, yv, zw, x^2-vw)$.
We get $(xy, xz, yz, yv, zw, x^2-vw)\supseteq (x,y,w)\cap(x,z,v)\cap(y,z,x^2-vw)$.
Since $(x,y,w)$, $(x,z,v)$, and $(y,z,x^2-vw)$ are prime ideals of height 3 (see Remark \ref{domain rem}), we see that $R$ is reduced, equidimensional, and of dimension 2.

\begin{claim}
$v+w,y+z+v$ is an $R$-regular sequence.
In particular, $R$ is Cohen--Macaulay.
\end{claim}

By Claim 1, $v+w$ is not in any associated prime ideal of $R$.
So, it is a non-zero-divisor of $R$.
There exists the natural ring isomorphism $R/(v+w)R\cong K \llbracket X, Y, Z, V \rrbracket/(XY, XZ, YZ, YV, ZV, X^2-V^2)$.
Following a similar reasoning as in the proof of Claim 1, the primary decomposition of $(XY, XZ, YZ, YV, ZV, X^2-V^2)$ is $(X,Y,V)\cap(X,Z,V)\cap(Y,Z,(X+V)^2)$ if $p=2$, and $(X,Y,V)\cap(X,Z,V)\cap(Y,Z,X+V)\cap(Y,Z,X-V)$ otherwise.
(In fact, if $p\ne 2$, for $AY+BZ+C(X-V)\in (X,Y,V)\cap(X,Z,V)\cap(Y,Z,X+V)$, where $A,B,C\in K \llbracket X, Y, Z, V \rrbracket$, we get $A\in (X,Z,V)$, $B\in (X,Y,V)$, and $C\in(Y,Z,X+V)$.)
We see that $y+z+v$ is a non-zero-divisor on $R/(v+w)R$.

\begin{claim}
$I:=(v+w,y+z+v)R$ is not tightly closed in $R$.
In particular, $R$ is not $F$-rational.
\end{claim}

The natural isomorphism $R/I\cong K \llbracket X, Y, Z\rrbracket/(XY, XZ, YZ, Y^2, Z^2, X^2)$ implies that $yR$ is not contained in $I$.
We have 
$(x-y+z)y^q=(xy)y^{q-1}+(yz)(y^{q-1}+z^{q-1})+(yv)v^{q-1}-(y+z+v)^q y$ for any $e>0$ and $q=p^e$.
Hence $(x-y+z)y^q R$ is contained in $I^{[q]}$, which deduces that $yR$ is contained in $I^\ast$ since $x-y+z \notin (x,y,w)\cup(x,z,v)\cup(y,z,x^2-vw)$ in $K \llbracket x, y, z, v, w \rrbracket$.

\begin{claim}
The socle of $R/I$ is not contained in $I^\ast/I$.
\end{claim}

Since $R/I\cong K \llbracket X, Y, Z\rrbracket/(XY, XZ, YZ, Y^2, Z^2, X^2)$, it is enough to show that $\m\nsubseteq I^\ast$.
We prove $(x+y)R \nsubseteq I^\ast$.
Let $\p=(y,z,x^2-vw)R$ and $S=K \llbracket X,V,W \rrbracket/(X^2-VW)\cong R/\p$.
If $(x+y)R \subseteq I^\ast$, then we have $(x+y)(R/\p) \subseteq (I(R/\p))^\ast$ since $\p$ is a minimal prime ideal of $R$.
This means $XS\subseteq ((V,W)S)^\ast$.
Since $S$ is strongly $F$-regular, $(V,W)S$ is tightly closed and contains $XS$, a contradiction.

\begin{claim}
The parameter test ideal of $R$ is not an $F$-ideal.
\end{claim}

We put $N=0_{\H_\m^2(R)}^\ast$.
The parameter test ideal $J$ of $R$ is equal to $\Ann_R N$ by \cite[Proposition 4.4(ii)]{S2} and Claim 2.
If $J$ is an $F$-ideal, the submodule $(0:_{\H_\m^2(R)}J)$ of $\H_\m^2(R)$ is $F$-stable and contains $N$.
Combining Claim 1, \cite[Proposition 2.5]{S3} (or \cite[Proposition 2.11]{EH}), and \cite[Theorem 6.1(a)]{HH}, we obtain $N=(0:_{\H_\m^2(R)}J)$.
By Claim 3, $J$ is a proper ideal and thus $N$ contains the socle of $\H_\m^2(R)$.
This is in contradiction with Claim 4.
(By the proof of Claim 4, $[\alpha+(v+w,y+z+v)R]\in \varinjlim_{n} R/((v+w)^n,(y+z+v)^n)R\cong \H_\m^2(R)$ belongs to the socle of $\H_\m^2(R)$ but does not belong to $N$, where $\alpha$ is the image of $x+y$ in $R$.)
\end{ex}

Note that the parameter test ideal and the conductor of the ring $(R,\m)$ in Example \ref{count. para. F-ideal CM} are not contained in the ideal generated by a system of parameters.
In fact, according to Macaulay2, the parameter test ideal of $R$ is $(x,y,z)R$.
If $(x,y,z)R\subseteq (f,g)R$ for some $f,g \in \m$, then the 3-dimensional vector space $(x,y,z)(R/\m^2)$ is contained in the vector space $(f,g)(R/\m^2)$ of dimension at most 2, which is a contradiction.

Next, we consider the existence of Cohen--Macaulay rings that give counterexamples to Questions \ref{ques 1} and \ref{ques 2}.
By Remark \ref{others remark}(5), such counterexamples are provided by rings of dimension greater than or equal to two if they exist.
We focused on rings with minimal multiplicity because the square of the maximal ideal of such rings is often contained in some ideal generated by a system of parameters; see \cite[Exercise 4.6.14]{BH}.
Therefore, it would be sufficient to compare the conductor and the parameter test ideal with the square of the maximal ideal.
In the following, we provide counterexamples of two-dimensional Cohen--Macaulay local rings with minimal multiplicity to \cite[Exercise 12.5]{HS} and \cite[Proposition 6.1]{S2}.

\begin{ex}\label{count. para. s.o.p}
Let $K$ be a field, $a,b$ positive integers, $R=K[x, y, z, w]/(x^az-y^2, x^aw^b-yz, w^by-z^2)$ a quotient of a polynomial ring over $K$, and $\m=(x, y, z, v, w)R$.

\begin{claim2}
$R$ is a two-dimensional Cohen--Macaulay ring, and $\m^2=(x,w)\m\subseteq (x,w)R$.
\end{claim2}
The latter part is obvious.
Consider the $2\times 3$ matrix $H=\left(\begin{smallmatrix}
   x^a & y & z\\
   y & z & w^b
\end{smallmatrix}\right)$
that is similar to the Hankel matrix.
For $S=K[x, y, z, w]$, it is well-known that $R=S/I_2(H)$ is a Cohen--Macaulay ring of dimension 2. 
Alternatively, using \cite[Corollary 1]{BE}, it is easy to see that the complex
$$
0\to S^{\oplus 2} \xrightarrow{H^\perp} S^{\oplus 3} \xrightarrow{(w^by-z^2\ -x^aw^b+yz\ x^az-y^2)} S \to R\to 0
$$
is exact, and the claim is derived by applying the Auslander-Buchsbaum formula to the localization of the $S$-module $R$.

\begin{claim2}
$R$ is an integral domain if at least one of $a$ or $b$ is not divisible by 3.
In particular, in this case, $R_\m$ is an analytically unramified Cohen--Macaulay local domain.
\end{claim2}
By Claim 1, $x$ is a non-zero-divisor of $R$.
Set $T= K[X, Y, W]/(X^{2a}W^b-Y^3)$.
There are isomorphisms
\begin{align*}
R_x &\cong K[x, y, z, w,x^{-1}]/(z-y^2x^{-a}, x^aw^b-yz, w^by-z^2) \\
&\cong K[X, Y, W, X^{-1}]/(X^aW^b-Y^3X^{-a}, W^bY-Y^4X^{-2a}) \\
&= K[X, Y, W, X^{-1}]/(X^{2a}W^b-Y^3)\cong T_X.
\end{align*}
If at least one of $a$ or $b$ is not divisible by 3, $T$ is an integral domain by Remark \ref{domain rem}, and hence so are $T_X$, $R_x$, $R$, and $R_\m$.
Since $R$ is excellent, the local domain $R_\m$ is analytically unramified.

\begin{claim2}
The conductor $\mathfrak{C}$ of $R_\m$ is contained in the parameter ideal $(x,w)R_\m$ if $a,b\ge 3$.
In particular, the parameter test ideal of $R_\m$ is contained in $(x,w)R_\m$ if $R$ has positive characteristic.
\end{claim2}

The latter assertion follows from Remark \ref{conduc para test}.
We prove the former.
We denote the image of $f\in S$ in $R$ by $[f]$.
Let $\overline{R_\m}$ be the integral closure of $R_\m$ in the fraction field $R_{(0)}$.
Put $c={\rm min}\{n\in \mathbb{Z} \mid 2a\le 3n\}$ and $d={\rm min}\{n\in \mathbb{Z} \mid 2b\le 3n\}$.
The equalities
$$
\left(\frac{[x^aw^d]}{[z]}\right)^3 \hspace{-5pt}=\frac{[x^{3a}w^{3d}]}{[z^3]}=\frac{[x^{2a}w^{3d-b}y]}{[z^2]}=[x^{2a}w^{3d-2b}], 
\left(\frac{[x^cw^b]}{[y]}\right)^3 \hspace{-5pt}=\frac{[x^{3c}w^{3b}]}{[y^3]}=\frac{[x^{3c-a}w^{2b}z]}{[y^2]}=[x^{3c-2a}w^{2b}]
$$
hold in $R_{(0)}$.
So $[x^aw^d]/[z]$ and $[x^cw^b]/[y]$ belong to $\overline{R_\m}$.
Let $\alpha/\beta\in \mathfrak{C}$, where $\alpha\in R$ and $\beta\in R\setminus\m$.
We can write $\alpha=[A+Bx+Cy+Dz+Ew+F]$ for some $A,B,C,D,E \in K$ and $F\in(x,y,z,w)^2$.
As $\alpha/\beta\in \mathfrak{C}$, we obtain $(\alpha/\beta)([x^aw^d]/[z]) \in R_\m$, and hence $\gamma \alpha [x^aw^d] \in \beta[z] R\subseteq zR$ for some $\gamma\in R\setminus \m$.
One can write $\gamma=[G+H]$ for some $0\ne G\in K$ and $H\in (x,y,z,w)$.
This implies that
$$
(G+H)(A+Bx+Cy+Dz+Ew+F)x^aw^d \in (z, y^2, x^aw^b, w^by)
$$
in $S$.
Since $3(d-1)<2b$ and $b\ge 3$, we know $d<(2b+3)/3\le b$.
So, it is clear that the coefficients of $x^aw^d$, $x^{a+1}w^d$, and $x^ayw^d$ for any element of $(z, y^2, x^aw^b, w^by)$ must be zero.
Therefore $A=B=C=0$ as $G\ne 0$.
Due to symmetry, $(\alpha/\beta)([x^cw^b]/[y]) \in R_\m$ derives $D=E=0$.
We conclude that $\alpha/\beta=[F]/\beta\in (x,w)R_\m$ by Claim 1.
\end{ex}

We conclude this section with a few remarks.

\begin{rem}\label{others remark}
(1) The complete local ring in Example \ref{count. para. F-ideal CM} is isomorphic to the $(x_1,x_2,x_3,x_4,x_5)$-adic completion of the ring in \cite[Section 9]{Kat} when $K$ is the field of two elements.
In fact, a homomorphism $\phi:K\llbracket x_1,x_2,x_3,x_4,x_5\rrbracket \to K\llbracket x,y,z,v,w\rrbracket$ of $K$-algebra with $\phi(x_1)=x, \phi(x_2)=x+y, \phi(x_3)=z+v, \phi(x_4)=v, \phi(x_5)=w$ induces an isomorphism between these quotient rings.

(2) Consider an example of a Cohen--Macaulay complete local domain for Example \ref{count. para. s.o.p}.
Let $S=K \llbracket x, y, z, w \rrbracket/(x^az-y^2, x^aw^b-yz, w^by-z^2)$ be a quotient of a formal power series ring over a field $K$.
Claims 1 and 3 of Example \ref{count. para. s.o.p} hold for $S$ for the same reason.
However, a different argument from Claim 2 is required to know whether $S$ is an integral domain.
Suppose that at least one of $a$ or $b$ is not divisible by 3 and that $R$ is as in Example \ref{count. para. s.o.p}.
Let $n$ be either 1 or 2.
Since $R$ is Cohen--Macaulay, $x^n-w, x$ is an $R$-regular sequence.
We see that $R/(x^n-w)R\cong K[X,Y,Z]/(X^aZ-Y^2, X^{a+nb}-YZ, X^{nb}Y-Z^2)$ and that there is a natural surjection $\phi_n: R/(x^n-w)R \to K[t^3, t^{2a+nb}, t^{a+2nb}] \subseteq K[t]$.
A similar argument to the proof of Claim 2 shows that $(R/(x^n-w)R)_x\cong (K[X,Y]/(X^{2a+nb}-Y^3))_X$.
By Remark \ref{domain rem}, $R/(x^n-w)R$ is an integral domain for some $n\in\{1,2\}$.
Then, since $\phi_n$ is a surjection of one-dimensional domains, it is an isomorphism.
By completion, we obtain $S/(x^n-w)S \cong K\llbracket t^3, t^{2a+nb}, t^{a+2nb}\rrbracket\subseteq K\llbracket t\rrbracket$, which implies that $S/(x^n-w)S$ is an integral domain, and thus so is $S$; see \cite[Claim 2 of Remark 4.1]{KMN}.

(3) Conversely, in (2) and Example \ref{count. para. s.o.p}, $S$ and $R_\m$ are not integral domains if both $a$ and $b$ are divisible by 3.
Indeed, if $a=3m$ and $b=3n$ for some integers $m,n$, we have $(x^m z-w^n y)(x^{2m}w^{2n}+x^m z+w^n y)=w^{2n}(x^az-y^2)-x^{2m}(w^by-z^2)$.
It is easy to see that the images of $x^m z-w^n y$ and $x^{2m}w^{2n}+x^m z+w^n y$ in $S$ and $R_\m$ are non-zero, but their product is zero.
Moreover, it can be inferred by Macaulay2 that the conductor and the parameter test ideal of $S$ and $R_\m$ are not contained in any ideal generated by a system of parameters if $a\le 2$ or $b\le 2$.
For example, according to Macaulay2, the conductor of $R$ contains $(z,w^e y)R$ for some $0\le e<b$ when $a\le 2$.
Now we show that $(z,w^e y)S$ is not contained in any ideal generated by a system of parameters.
(The same holds for $R_\m$.)
Suppose that there are power series $f,g$ whose images in $S$ form a system of parameters, such that $(z,w^e y)S\subseteq (f,g)S$.
There exist some power series $p.q$ such that the images of $pf+qg$ and $z$ in $S$ are equal.
Since $z$ is not in $(x,y,z,w)^2$, neither is $pf+qg$, which means that $p$ or $q$ is a unit.
If $q$ is a unit, $(f,g)S=(f,z)S$ holds.
Hence we may assume that $f\in (x,y,w)A$ and $g=z$, where $A=K\llbracket x,y,w\rrbracket$.
Then we have $f\notin (x,y)A\cup (w,y)A$ as $f$ is $S/zS$-regular, and $(z,w^e y)S\subseteq (f,g)S$ implies $w^e y\in (f, y^2, x^aw^b, w^by)A$.
We can write $w^e y=\alpha f+\beta$ for some $\alpha\in A$ and $\beta\in (y^2, x^aw^b, w^by)A$.
The image of $\alpha f$ in $A/(y^2, x^aw^b, w^ey)A$ is zero and $(y^2, x^aw^b, w^ey)A=(y,x^a)A\cap (y^2, w^e y, w^b)A$ is a primary decomposition.
Since $f\notin (x,y)A\cup (w,y)A$, $f$ is $A/(y^2, x^aw^b, w^ey)A$-regular.
This deduces $\alpha\in (y^2, x^aw^b, w^ey)A$ and thus $w^e y \in w^ey(x,y,w)A+(y^2, x^aw^b, w^by)A$, which is a contradiction.

(4) The reason why the proof of \cite[Proposition 4.5]{S2} did not work is that it is not necessarily true that $(0:_{\H_\m^d(R)}J)=0_{\H_\m^d(R)}^\ast$ holds for the parameter test ideal $J$.
If $R$ is complete and quasi-Gorenstein, then $\H_\m^d(R)$ is the injective hull of the residue field of $R$ and $J$ is equal to $\Ann_R (0_{\H_\m^d(R)}^\ast)$.
Thus the above equality holds.
This implies that $J$ is an $F$-ideal.
Moreover, when $R$ is reduced, $J$ is non-zero by \cite[Theorem 6.1(a)]{HH}.
Therefore, Smith proved in \cite{S2} that the parameter test ideal of a Gorenstein reduced complete local ring is not contained in any parameter ideal.

(5) Let $R$ be an analytically unramified local ring of dimension 1.
We see that the conductor $\mathfrak{C}$ is not contained in any parameter ideal.
In fact, let $K$ be the total ring of fraction, and $\overline{R}$ the integral closure of $R$.
If $\mathfrak{C}$ is contained in $cR$ for $c\in R$, then we have $1/c\in (R:_K cR)\subseteq (R:_K \mathfrak{C})=\overline{R}$ (see the proof of \cite[Corollary 2.6]{M} for the last equation).
One has the equality $\overline{R}=c\overline{R}$.
Since $R$ is analytically unramified, $\overline{R}$ is finitely generated over $R$.
By Nakayama's lemma, $c$ is a unit. 
In particular, by Remark \ref{conduc para test}, the parameter test ideal of a one-dimensional excellent reduced local ring is not contained in any parameter ideal, that is, \cite[Proposition 6.1]{S2} holds true.
\end{rem}

\section{Dualizing complex and standard conditions}

In this section, following the approach of Dey and Dutta \cite{DD}, we consider Questions \ref{ques 1} and \ref{ques 2} by studying the trace ideal of big Cohen--Macaulay modules.
While following the ideas of their proof, we aim to generalize some previous results by extending the argument to derived functors.
The notation below is used in this section.

\begin{conv}\label{def} 
Let $(R,\m)$ be a local ring of dimension $d$.

(1) Assume all complexes of $R$-modules are chain complex.
For a complex $X$ and an integer $m$, $\Sigma^{m}X$ denotes the complex $X$ shifted $m$ degrees, that is, $(\Sigma^{m}X)_i=X_{i-m}$.
We assume that the components of a dualizing complex of $R$ are non-zero only in the range $[0,-d]$ and that the components of the Koszul complex of a sequence $x_1,\ldots,x_n$ in $R$ are non-zero only in the range $[n,0]$.
An $R$-module $M$ is called a \textit{big Cohen--Macaulay module} if there is a system of parameters that is an $M$-regular sequence.

(2) Let $\alpha:F\to G$ be a homomorphism of free $R$-modules.
The \textit{rank} of $\alpha$ is the largest integer $r$ such that the induced map $\bigwedge^r \alpha: \bigwedge^r F\to \bigwedge^r G$ is not zero, and is denoted by $\rank \alpha$.
We denote by $I_t(\alpha)$ the ideal generated by the $t$-minors of $\alpha$.
Then $\rank \alpha$ is the same as the largest integer $t$ such that $I_t(\alpha)\ne 0$.

(3) Let $G$ be a complex of finitely generated free $R$-modules
$$
0\to G_{n}\xrightarrow{\alpha_{n}} G_{n-1} \xrightarrow{\alpha_{n-1}} \cdots \xrightarrow{\alpha_{2}} G_{1} \xrightarrow{\alpha_{1}} G_{0}\to 0.
$$
Let $b_i$ be the rank of $G_i$ and $r_i=\sum_{j=i}^{n} (-1)^{j-i}b_j$ for each $1\le i\le n$.
We say that $G$ satisfies the \textit{standard condition on rank} if $\rank \alpha_i=r_i$ for any $1\le i\le n$ (equivalently, $b_i=\rank\alpha_i+\rank\alpha_{i+1}$ for any $1\le i\le n$), and that $G$ satisfies the \textit{standard condition for height} 
if the height 
of the ideal $I_{r_i}(\alpha_i)$ is at least $i$ for each $1\le i\le n$.

(4) Assume that $R$ is of prime characteristic.
Let $G$ be a complex as in (3).
Following Aberbach \cite{Ab} and Hochster-Huneke \cite{HH2}, the $i$-th homology $H_i(G)$ is said to be \textit{phantom} if $\ker \alpha_i  \subseteq (\image \alpha_{i+1})^\ast_{G_i}$ holds.
We say that $G$ is \textit{phantom acyclic} if $H_i(G)$ is phantom for all $i>0$, and that it is \textit{stably phantom acyclic} if $F^e(G)$ is phantom acyclic for all $e>0$.
We also say that an $R$-module $M$ has \textit{finite phantom projective dimension} if there is a stably phantom acyclic complex $G$ such that $H_0(G)\cong M$.
If $R$ is reduced and $G$ is stably phantom acyclic, then $G$ satisfies the standard condition on rank and height; see \cite[Theorem 1.4.9]{Ab} or \cite[Theorem 9.8]{HH2}.

(5) For $R$-modules $M,N$, we set $\tr_N (M)=\sum_{f\in\Hom_R(M,N)} f(M)$.
Let $E$ be the injective hull of $R/\m$.
If $\H_\m^d(R)\cong E$, $R$ is said to be \textit{quasi-Gorenstein}.
A finitely generated $R$-module $\omega$ is called a \textit{canonical module} if $\Hom_R(\omega, E)\cong \H_\m^d(R)$.
We have $\omega\cong \Hom_R(\H_\m^d(R), E)$ when $R$ is complete, and $\omega\cong R$ when $R$ is complete and quasi-Gorenstein.
Additionally, let $R$ be a complete local domain of prime characteristic.
Applying the Matlis dual to the natural surjection $\H_\m^d(R)\to \H_\m^d(R)/0_{\H_\m^d(R)}^\ast$, $\Hom_R(\H_\m^d(R)/0_{\H_\m^d(R)}^\ast, E)$ can be regarded as a submodule of $\omega$.
It is called the \textit{parameter test submodule} and is denoted by $\tau (\omega)$.
Note that this consistent with the definition in \cite[Definition-Proposition 2.7]{MS}, and coincides with the parameter test ideal when $R$ is quasi-Gorenstein; see the proof of Corollary \ref{nonCM cor}.
\end{conv}

\begin{rem}\label{std cond and para or fpd}
Let $R$ be a local ring, $F$ a (minimal) free resolution of a finitely generated $R$-module with finite projective dimension, $\boldsymbol{x}=x_1,\ldots,x_s$ a sequence in $R$ such that $\height\boldsymbol{x}R=s$, and $K$ the Koszul complex $K$ of $\boldsymbol{x}$.
It is known (or easily shown) that $F$ and $K$ satisfy the standard conditions on rank and height.
Furthermore, if $R$ has prime characteristic, and $I$ is a proper ideal of $R$ that either has finite projective dimension or is a parameter ideal, then $R/I$ has finite phantom projective dimension.
However, none of those facts are required in this paper.
Of course, (3) and (4) of Proposition \ref{fpd lemma}, and (2) and (3) of Corollaries \ref{nonCM original} and \ref{nonCM cor}, can be omitted.
However, if limited to those cases, it can be seen that the proofs are more refined and compact.
(In fact, the above mentioned facts, Lemmas \ref{simple minimal lemma} and \ref{std cond lemma}, are not required.)
From this perspective,  this section treats them as independent conditions.
\end{rem}

In the following, we prepare some basic lemmas that will be used in this section.
Some of these are well-known facts, but here we provide proofs for the benefit of the reader.

\begin{lem}\label{dual cpx remark}
Let $R$ be a local ring with a dualizing complex $\dual$.
Then the following hold:
\begin{enumerate}[\rm(1)]
\item $R$ is catenary.
\item Let $\p$ be a prime ideal of $R$, and $D'$ a dualizing complex of $R_\p$.
Then there is a quasi-isomorphism $D_\p\simeq \Sigma^{-n}D'$, where $n=\dim R-\dim R/\p-\height\p$.
\end{enumerate}
\end{lem}

\begin{proof}
Since $R$ admits a dualizing complex, it is a homomorphic image of a Gorenstein local ring $S$ such that $\dim S=\dim R$; see \cite[Corollary 1.4]{Kaw}.
Hence (1) holds.
Then $D\simeq\RHom_S(R,S)$.
Put $R=S/I$.
Let $\q$ be the inverse image of $\p$ in $S$.
Since $S$ is Cohen--Macaulay, we have the following equalities:
$$
\grade(IS_\q,S_\q)=\height IS_\q=\dim S_\q-\dim R_\p=\dim S-\dim S/\q-\dim R_\p=n. 
$$
We can choose an $S_\q$-regular sequence $x_1,\ldots,x_n$ in $I$ and set $A=S_\q/(x_1,\ldots,x_n)S_\q$.
We know $\Sigma^{-n} A\simeq\RHom_{S_\q}(A,S_\q)$.
Since $R_\p$ is a homomorphic image of $A$, which is a Gorenstein local ring such that $\dim A=\dim S_\q-n=\dim R_\p$, we obtain $D'\simeq\RHom_A(R_\p,A)$.
There are quasi-isomorphisms
$$
\RHom_{S_\q}(R_\p,S_\q)\simeq\RHom_A(R_\p, \RHom_{S_\q}(A,S_\q))\simeq\RHom_A(R_\p, \Sigma^{-n} A)\simeq\Sigma^{-n}\RHom_A(R_\p,A)
$$
which induce $D_\p\simeq \Sigma^{-n}D'$.
\end{proof}

\begin{lem}\label{simple minimal lemma}
Let $R$ be a local ring, and $F=(0\to F_{n}\to \cdots \to F_{0}\to 0)$ a complex of finitely generated free $R$-modules.
Then there exists a complex of finitely generated free $R$-modules $H=(0\to H_{n}\to \cdots \to H_{0}\to 0)$ such that $F\simeq H$ and the minimal number of generators of $H_0(F)$ is $\rank H_0$.
\end{lem}

\begin{proof}
The case $n=0$ is trivial. Assume $n>0$.
Let $\mu$ the minimal number of generators of $H_0(F)$.
Applying $(-)\otimes_R k$ to the exact sequence $F_1\xrightarrow{\alpha} F_0\to H_0(F)\to 0$, we obtain an exact sequence $k^{\oplus a}\xrightarrow{\alpha\otimes k} k^{\oplus b}\to k^{\oplus \mu}\to 0$, where $k$ is the residue field of $R$, $a=\rank F_1$, and $b=\rank F_0$.
The dimension of $\image(\alpha\otimes k)$ as a $k$-vector space is $b-\mu$, so $I_{b-\mu}(\alpha)=R$.
By elementary row and column operations, we obtain the following commutative diagram:
\[
  \xymatrix@C=40pt@R=20pt{
    F_1 \ar[r]^{\alpha} \ar[d]_{\rotatebox{90}{$\cong$}}
    & F_0 \ar[d]_{\rotatebox{90}{$\cong$}}  \\
    R^{\oplus(b-\mu)} \oplus R^{\oplus(a-b+\mu)} \ar[r]^{\quad \begin{pmatrix}
   \operatorname{id} & 0 \\
   0 & \beta
\end{pmatrix}}
    & R^{\oplus(b-\mu)} \oplus R^{\oplus\mu}, \\
  }
\]
 which induces a short exact sequence 
\[
  \xymatrix@C=15pt@R=10pt
  {
    0 \simeq (0 \ar[r] 
    & 0 \ar[r] \ar[d]
    & \cdots \ar[r]
    & 0 \ar[r] \ar[d]
    & R^{\oplus(b-\mu)} \ar@{=}[r] \ar[d]
    & R^{\oplus(b-\mu)} \ar[r] \ar[d]
    & 0 )\\
    F = (0 \ar[r]
    & F_n \ar[r]  \ar@{=}[d]
    & \cdots \ar[r]
    & F_2 \ar[r]  \ar@{=}[d]
    & F_1 \ar[r]^{\alpha}    \ar[d]
    & F_0 \ar[r]  \ar[d]
    & 0 )\\
    H := (0 \ar[r]
    & F_n \ar[r] 
    & \cdots \ar[r]
    & F_2 \ar[r]
    & R^{\oplus(a-b+\mu)} \ar[r]^{\ \beta}
    & R^{\oplus\mu} \ar[r] 
    & 0 )\\
  }
\]
of complexes of finitely generated free $R$-modules.
Then $F\simeq H$.
The proof is now completed.
\end{proof}

\begin{lem}\label{std cond lemma}
Let $R$ be a local ring of dimension $d$, and $G$ a complex of finitely generated free $R$-modules satisfying the standard conditions on rank and height.
\begin{enumerate}[\rm(1)]
\item There exists a complex of finitely generated free $R$-modules $H$ satisfying the standard conditions on rank and height such that $G\simeq H$, $H_i=0$ for any $i>d$.
\item  For any prime ideal $\p$ of $R$, $G_\p$ satisfies the standard conditions on rank and height.
\end{enumerate}
\end{lem}

\begin{proof}
We write $G=(0\to G_{n}\xrightarrow{\alpha_{n}} G_{n-1} \xrightarrow{\alpha_{n-1}} \cdots \xrightarrow{\alpha_{2}} G_{1} \xrightarrow{\alpha_{1}} G_{0}\to 0)$.
Let $b_i$ be the rank of $G_i$ and $r_i=\sum_{j=i}^{n} (-1)^{j-i}b_j$ for each $1\le i\le n$.

(1) Perform the following operation if $n>d$.
By assumption, $\height I_{r_n}(\alpha_{n})\ge n>d$ and $\rank \alpha_n=r_n=b_n$.
By this, we get $I_{b_n}(\alpha_n)=R$ and $I_{b_n+1}(\alpha_n)=0$.
It follows from \cite[Proposition 1.4.10]{BH} that $\cok(\alpha_n)$ is a free $R$-module of rank $b_{n-1}-b_n$ and hence $\image(\alpha_n)$ is a free $R$-module of rank $b_n$.
The natural surjection $G_n\to \image(\alpha_n)$ is an isomorphism.
There is a short exact sequence 
\[
  \xymatrix@C=15pt@R=10pt
  {
    0 \simeq (0 \ar[r]
    &  G_n \ar[r] \ar@{=}[d]
    & \image(\alpha_n) \ar[r] \ar[d]    
    & 0 \ar[r] \ar[d]
    & \cdots \ar[r]
    & 0 \ar[r] \ar[d]
    & 0 \ar[r] \ar[d]
    & 0 )\\
    G = (0 \ar[r]
    & G_n \ar[r] \ar[d]
    & G_{n-1} \ar[r] \ar[d]
    & G_{n-2} \ar[r] \ar@{=}[d]
    & \cdots \ar[r]
    & G_1 \ar[r] \ar@{=}[d]
    & G_0 \ar[r] \ar@{=}[d]
    & 0 )\\
    H := (0 \ar[r]
    & 0 \ar[r] 
    & \cok(\alpha_n) \ar[r]^{\beta}
    & G_{n-2} \ar[r]
    & \cdots \ar[r]
    & G_1 \ar[r]
    & G_0  \ar[r] 
    & 0 )\\
  }
\]
of complexes of finitely generated free $R$-modules.
Then $G\simeq H$.
Since $\alpha_{n-1}$ and $\beta$ have the same image, they also have the same cokernel.
Hence $I_t(\alpha_{n-1})=I_t(\beta)$ for every integer $t$ because these are Fitting invariants of the same module.
We obtain $\rank \alpha_{n-1}=\rank \beta$.
It is easy to see that $H$ satisfies the standard conditions on rank and height.

(2) Fix $1\le i\le n$ and $\p\in\spec(R)$.
We have $\height I_{r_i}(\alpha_i)\ge i$ and $I_{r_i+1}(\alpha_i)=0$ by assumption.
Then $\height I_{r_i}((\alpha_i)_\p)=\height I_{r_i}(\alpha_i)R_\p\ge \height I_{r_i}(\alpha_i)\ge i$ and hence $I_{r_i}((\alpha_i)_\p)\ne 0$ as $i>0$.
On the other hand, $I_{r_i+1}((\alpha_i)_\p)=\height I_{r_i+1}(\alpha_i)R_\p=0$, which means $\rank (\alpha_i)_\p=r_i$.
We see that $G_\p$ satisfies the standard conditions on rank and height.
\end{proof}

The following proposition plays an essential role in the proof of the main result of this section.
This is a generalization of the Tor-orthogonality between maximal Cohen--Macaulay modules and modules of finite projective dimension in Cohen--Macaulay rings; see \cite[Theorem 5.3.10]{C} and Remark \ref{MCMcpx fpd}.

\begin{prop}\label{fpd lemma}
Let $R$ be a local ring with a dualizing complex $\dual$, and $N$ a finitely generated $R$-module.
\begin{enumerate}[\rm(1)]
\item Let $X$ be a complex of $R$-modules.
Suppose that for any prime ideal $\p$ of $R$, $X_\p$ is a quasi-isomorphic to a complex $P$ of finitely generated free $R_\p$-modules such that $P_i=0$ for any $i>\height\p$.
Then for any $i>0$, $H_i(X\otimes_R^\L \RHom_R(N,\dual))=0$ holds.
In particular, $H_i(X\otimes_R^\L\dual)=0$ for every $i>0$.
\item If $G$ is a complex of finitely generated free $R$-modules satisfying the standard conditions on rank and height, $H_i(G\otimes_R^\L \RHom_R(N,\dual))=0$ for any $i>0$.
In particular, $H_i(G\otimes_R^\L\dual)=0$ for every $i>0$.
\item Suppose that $M$ is a finitely generated $R$-module of finite projective dimension.
Then for any $i>0$, $H_i(M\otimes_R^\L \RHom_R(N,\dual))=0$ holds.
In particular, $H_i(M\otimes_R^\L\dual)=0$ for every $i>0$.
\item Let $\boldsymbol{x}=x_1,\ldots, x_s$ be a sequence in $R$ such that $\height  \boldsymbol{x}R=s$, and $K$ the Koszul complex of $\boldsymbol{x}$.
Then for any $i>0$, $H_i(K\otimes_R^\L \RHom_R(N,\dual))=0$ holds.
In particular, $H_i(K\otimes_R^\L\dual)=0$ for every $i>0$.
\end{enumerate}
\end{prop}

\begin{proof}
(1) First, we consider the case where $H_i(X\otimes_R^\L\RHom_R(N,\dual))$ has finite length for any $i>0$.
Then we see that $H_i(X\otimes_R^\L\RHom_R(N,\dual))\cong H_i(X\otimes_R^\L\RHom_R(N,\dual)\otimes_R^\L C)$ for any $i>0$ where $C$ is the \v{C}ech complex on a system of parameters of $R$.
Let $E$ be the injective hull of the residue field of $R$ and $d=\dim R$.
Since $\dual\otimes_R^\L C\simeq\Sigma^{-d} E$, we get $\RHom_R(N,\dual)\otimes_R^\L C\simeq\RHom_R(N,\dual\otimes_R^\L C)\simeq\Sigma^{-d}\Hom_R(N,E)$.
Now $X$ is a quasi-isomorphic to a complex $P$ of finitely generated free $R$-modules such that $P_i=0$ for any $i>d$.
We have $H_i(X\otimes_R^\L\RHom_R(N,\dual))\cong H_i(P\otimes_R \Sigma^{-d}\Hom_R(N,E))=0$ for any $i>0$.

Next, we address the general case.
We put $V_n=\bigcup_{i>-n} \supp(H_i(X\otimes_R^\L\RHom_R(N,\dual)))$ for each integer $n$.
Suppose $V_0\ne \emptyset$.
We take a minimal element $\p_1$ of $V_0$ and put $n_1=\dim R-\dim R/\p_1-\height\p_1$.
Since $n_1\ge 0$, $\p_1$ belongs to $V_{n_1}$.
Repeat the following step.
If $\p_j$ is not minimal in $V_{n_j}$, we choose a minimal element $\p_{j+1}$ of $V_{n_j}$, which is contained in $\p_j$.
The equality $\dim R/\p_{j+1}=\dim R/\p_j+\height(\p_j/\p_{j+1})$ holds since $R$ is catenary and local; see Lemma \ref{dual cpx remark}(1).
We have 
$$
\dim R/\p_{j+1}+\height\p_{j+1}=\dim R/\p_j+\height(\p_j/\p_{j+1})+\height\p_{j+1}\le \dim R/\p_j+\height\p_j.
$$
We get $n_{j+1}:=\dim R-\dim R/\p_{j+1}-\height\p_{j+1}\ge n_j$, which means $\p_{j+1}\in V_{n_j}\subseteq V_{n_{j+1}}$.
This step can be repeated at most $\height\p_1$ times.
Therefore, there exists a prime ideal $\p$ of $R$ which is a minimal element of $V_{n}$, where $n=\dim R-\dim R/\p-\height\p$.
Let $\dual'$ be a dualizing complex of $R_\p$.
We see by Lemma \ref{dual cpx remark}(2) that $\dual_\p\simeq \Sigma^{-n}\dual'$ and thus $\RHom_R(N,\dual)_\p\simeq\Sigma^{-n}\RHom_{R_\p}(N_\p,\dual')$.
This implies that for all $i>0$, 
$$
H_i(X_\p\otimes_{R_\p}^\L\RHom_{R_\p}(N_\p,\dual'))\cong H_{i-n}(X_\p\otimes_{R_\p}^\L\RHom_R(N,\dual)_\p)\cong H_{i-n}(X\otimes_R^\L\RHom_R(N,\dual))_\p
$$
hold and these $R_\p$-modules have finite length since $\p$ is minimal in $V_n$.
By the previous paragraph, $H_{i}(X_\p\otimes_{R_\p}^\L\RHom_{R_\p}(N_\p,\dual'))=0$ for all $i>0$.
This contradicts the fact that $\p$ belongs to $V_n$.

The assertions (2), (3), and (4) are all consequences of (1).
Indeed, (2) follows from Lemma \ref{std cond lemma}.
In the case of (3), we have $\pd_{R_\p}M_\p \le \height \p$ for every prime ideal $\p$ of $R$.
In the case of (4), for any prime ideal $\p$ of $R$ such that $\height\p<s$, $\boldsymbol{x}R$ is not contained in $\p$ and thus $K_\p\simeq 0$.
\end{proof}

\begin{rem}\label{MCMcpx fpd}
Let $(R,\m,k)$ be a local ring.
A complex of $R$-modules $X$ is said to be \textit{maximal Cohen--Macaulay} if $\bigoplus_{i\in\mathbb{Z}} H_i(X)$ is finitely generated, $H_0(X)\to H_0(k\otimes_R^\L X)$ is non-zero, and $H^i_{\m}(X)=0$ for $i\ne \dim R$ (as a cochain complex).
By definition, any maximal Cohen--Macaulay module is a maximal Cohen--Macaulay complex.
If $R$ admits a dualizing complex $\dual$ and $X$ is a maximal Cohen--Macaulay complex, then $X\simeq\RHom_R(N,\dual)$, where $N=H_0(\RHom_R(X,\dual))$ by \cite[Proposition 4.3]{IMSW}.
Therefore, for $G$, $M$, and $K$ in Proposition \ref{fpd lemma}, the positive homologies of $G\otimes_R^\L X$, $M\otimes_R^\L X$, and $K\otimes_R^\L X$ all vanish.
\end{rem}

The next lemma corresponds to the Ext-orthogonality between maximal Cohen--Macaulay modules and modules of finite injective dimension; see \cite[Exercise 3,1,24]{BH} for instance.

\begin{lem}\label{bigCM lemma}
Let $R$ be a complete local ring with a dualizing complex $\dual$, and $M$ a big Cohen--Macaulay $R$-module.
Then $H_0(\RHom_R (M, \dual))\ne 0$ holds.
Let $F=(0\to F_a\to \cdots \to F_b\to 0)$ be a complex of finitely generated free $R$-modules.
Then $H_i(\RHom_R (M, F\otimes_R^\L\dual))=0$ for any $i<b$.
In particular, $H_i(\RHom_R (M, N\otimes_R^\L\dual))=0$ holds for any $i<0$ and any finitely generated $R$-module $N$ of finite projective dimension.
\end{lem}

\begin{proof}
Since $R$ is complete, $R$ is a homomorphic image of a complete Gorenstein local ring $S$ such that $\dim S=\dim R$.
Then $D\simeq\RHom_S(R,S)$.
By \cite[Lemma 3.6]{DD}, we obtain $H_0(\RHom_R(M, \dual))\cong H_0(\RHom_S(M, S))\cong\Hom_S(M, S)\ne 0$.
We use induction on $a-b$ to prove the latter part.
If $a-b=0$, then $F_b\cong R^{\oplus n}$ for some $n$.
We have $\RHom_R(M, F\otimes_R^\L\dual)\simeq\Sigma^{b}\RHom_R(M, \dual)^{\oplus n}\simeq\Sigma^{b}\RHom_S(M, S)^{\oplus n}$ and thus $H_i(\RHom_R (M, F\otimes_R^\L\dual))\cong\Ext_S^{b-i}(M,S)^{\oplus n}=0$ for any $i<b$ by \cite[Lemma 3.6]{DD}.
Suppose that $a-b>0$.
There is a short exact sequence 
\[
  \xymatrix@C=15pt@R=10pt
  {
    G := ( 
    0 \ar[r]
    & 0 \ar[r] \ar[d]
    & \cdots \ar[r]
    & 0 \ar[r] \ar[d]
    & F_b \ar[r] \ar@{=}[d]
    & 0 )\\
    F = (0 \ar[r]
    & F_a \ar[r] \ar@{=}[d]
    & \cdots \ar[r]
    & F_{b+1} \ar[r] \ar@{=}[d]
    & F_b \ar[r] \ar[d]
    & 0 )\\
    H := (0 \ar[r]
    & F_a \ar[r]
    & \cdots \ar[r]
    & F_{b+1} \ar[r]
    & 0  \ar[r] 
    & 0 )\\
  }
\]
of complexes of finitely generated free $R$-modules.
For each $i<b$, this induces an exact sequence
$$
H_i(\RHom_R (M, G\otimes_R^\L\dual))\to H_i(\RHom_R (M, F\otimes_R^\L\dual))\to H_i(\RHom_R (M, H\otimes_R^\L\dual)).
$$
By the induction hypothesis, we have $H_i(\RHom_R (M, G\otimes_R^\L\dual))=0=H_i(\RHom_R (M, H\otimes_R^\L\dual))$, which implies $H_i(\RHom_R (M, F\otimes_R^\L\dual))=0$.
\end{proof}

We are now ready to prove the main result of this section.
In general, the canonical module is neither maximal Cohen--Macaulay nor of finite injective dimension. 
However, we can refine the proof in \cite{DD} by using the fact that it is the 0-th homology of the dualizing complex.

\begin{thm}\label{nonCM}
Let $R$ be a complete local ring, $I$ a proper ideal of $R$, $M$ a big Cohen--Macaulay $R$-module, $\dual$ a dualizing complex of $R$, and $\omega$ a canonical module of $R$.
Suppose that there exist a bounded complex $F$ of finitely generated free $R$-modules such that $H_i(F\otimes_R^\L \dual)=0$ for any $i>0$, $H_0(F)=R/I$, and $F_j=0$ for all $j<0$.
Then $\tr_\omega (M)\nsubseteq I\omega$.
In particular, if $R$ is quasi-Gorenstein, then $\tr_R (M)\nsubseteq I$.
\end{thm}

\begin{proof}
By Lemma \ref{simple minimal lemma}, we may assume $F_0=R$.
There is an exact sequence $F_1\xrightarrow{f} F_0\xrightarrow{\pi} R/I\to 0$ such that $\image f=I$.
Following the proof of \cite[Theorem 3.1]{DD}, we only need to prove that the sequence 
$$
\alpha: \Hom_R(M,F_1\otimes_R\omega)\xrightarrow{\Hom_R(M,f\otimes_R\omega)} \Hom_R(M,F_0\otimes_R\omega) \xrightarrow{\Hom_R(M,\pi\otimes_R\omega)} \Hom_R(M,R/I\otimes_R\omega)
$$ 
is exact and that $\Hom_R(M,\omega)$ is non-zero.
In fact, assume that $\alpha$ is exact. Then we obtain
$$
I \Hom_R(M,\omega)=\image (\Hom_R(M,f\otimes_R\omega))=\ker (\Hom_R(M,\pi\otimes_R\omega))=\Hom_R(M,I\omega).
$$
If $\tr_\omega (M)\subseteq I\omega$, then the equalities $\Hom_R(M,\omega)=\Hom_R(M,\tr_\omega (M))=\Hom_R(M,I\omega)$ hold, which implies $\Hom_R(M,\omega)=0$ by \cite[Lemma 2.8]{DD}.
As $R$ is complete, we get $\omega\cong H_0(D)$, and hence $\Hom_R(M,\omega)\cong\ker(\Hom_R(M,D_0)\to\Hom_R(M,D_{-1}))=H_0(\RHom_R (M, \dual))\ne 0$ by Lemma \ref{bigCM lemma}.
Hereafter, we show that $\alpha$ is exact.

The complex $F$ is the mapping cone of the following homomorphism of complexes:
\[
  \xymatrix@C=20pt@R=10pt
  {
    G := ( 
    0 \ar[r]
    & F_n \ar[r] \ar[d]
    & \cdots \ar[r]
    & F_2 \ar[r] \ar[d]
    & F_1 \ar[r] \ar[d]^{f}
    & 0 )
\\
    F_0 = (0 \ar[r]
    & 0 \ar[r] 
    & \cdots \ar[r]
    & 0 \ar[r]
    & F_0 \ar[r]
    & 0 ).
\\
  }
\]
There is a short exact sequence
\[
  \xymatrix@C=20pt@R=10pt
  {
    F_1 = ( 
    0 \ar[r]
    & 0 \ar[r] \ar[d]
    & \cdots \ar[r]
    & 0 \ar[r] \ar[d]
    & F_1 \ar[r] \ar@{=}[d]
    & 0 )
\\ G = (0 \ar[r]
    & F_n \ar[r] \ar@{=}[d]
    & \cdots \ar[r]
    & F_2 \ar[r] \ar@{=}[d]
    & F_1  \ar[r] \ar[d]
    & 0 )
\\ H := (0 \ar[r]
    & F_n \ar[r]
    & \cdots \ar[r]
    & F_2 \ar[r]
    & 0  \ar[r] 
    & 0 )
\\
  }
\]
of complexes of finitely generated free $R$-modules.
By assumption, $H_i(F\otimes_R^\L\dual)=0=H_i(F_0\otimes_R^\L\dual)$ for any $i>0$.
The exact triangle $G\to F_0\to F\to \Sigma^1 G$ in the derived category induces an exact sequence
$$
0\to H_0(G\otimes_R^\L\dual)\to H_0(F_0\otimes_R^\L\dual)\to H_0(F\otimes_R^\L\dual)
$$ and equalities $H_i(G\otimes_R^\L\dual)=0$ for all $i>0$.
We see that
$$
0\to \Hom_R(M, H_0(G\otimes_R^\L\dual))\xrightarrow{\psi} \Hom_R(M, H_0(F_0\otimes_R^\L\dual))\to \Hom_R(M, H_0(F\otimes_R^\L\dual))
$$ 
is exact and we have a natural isomorphism $\Hom_R(M, H_0(G\otimes_R^\L\dual))\cong H_0(\RHom_R(M, G\otimes_R^\L\dual))$.
On the other hand, $H_i(F_1\otimes_R^\L\dual)=0$ holds for any $i>0$ and thus there is a natural isomorphism
$\Hom_R(M, H_0(F_1\otimes_R^\L\dual))\cong H_0(\RHom_R(M, F_1\otimes_R^\L\dual))$.
We obtain a commutative diagram
\[
  \xymatrix@C=20pt@R=15pt{
    \Hom_R(M, H_0(F_1\otimes_R^\L\dual)) \ar[r]^{\varphi} \ar[d]_{\rotatebox{90}{$\cong$}} 
    & \Hom_R(M, H_0(G\otimes_R^\L\dual)) \ar[d]_{\rotatebox{90}{$\cong$}} 
    & \\
    H_0(\RHom_R(M, F_1\otimes_R^\L\dual)) \ar[r] 
    & H_0(\RHom_R(M, G\otimes_R^\L\dual))  \ar[r] 
    & H_0(\RHom_R(M, H\otimes_R^\L\dual)) \\
  }
\]
with exact rows.
By Lemma \ref{bigCM lemma}, $H_0(\RHom_R(M, H\otimes_R^\L\dual))=0$.
So $\varphi$ is surjective.
We see that 
$$
\beta: \Hom_R(M, H_0(F_1\otimes_R^\L\dual))\xrightarrow{\psi\circ\varphi} \Hom_R(M, H_0(F_0\otimes_R^\L\dual))\to \Hom_R(M, H_0(F\otimes_R^\L\dual))
$$ 
is exact and that $\psi\circ\varphi=\Hom_R(M, H_0(f\otimes_R^\L\dual))$.
Combining the map $H_0(F_1\otimes_R^\L\dual)\to H_0(G\otimes_R^\L\dual)$ and the complex $H_0(G\otimes_R^\L\dual)\to H_0(F_0\otimes_R^\L\dual)\to H_0(F\otimes_R^\L\dual)$, we have a complex 
$$
H_0(F_1\otimes_R^\L\dual)\xrightarrow{H_0(f\otimes_R^\L\dual)} H_0(F_0\otimes_R^\L\dual)\to H_0(F\otimes_R^\L\dual).
$$
Since $\omega\cong H_0(D)$ and the $R$-modules $F_1$ and $F_0$ are free, we get a commutative diagram
\[
  \xymatrix@C=50pt@R=15pt{
    F_1 \otimes_R \omega \ar[r]^{f\otimes_R\omega} \ar[d]_{\rotatebox{90}{$\cong$}} 
    & F_0 \otimes_R \omega \ar[r]^{\pi\otimes_R\omega} \ar[d]_{\rotatebox{90}{$\cong$}} 
    & R/I \otimes_R \omega \ar[r] \ar[d] 
    & 0 \\
    H_0(F_1\otimes_R^\L\dual) \ar[r]^{H_0(f\otimes_R^\L\dual)}
    & H_0(F_0 \otimes_R^\L\dual) \ar[r] 
    & H_0(F \otimes_R^\L\dual) \\
  }
\]
by the universal property of the cokernel of $f\otimes_R\omega$.
This yields a commutative diagram 
\[
  \xymatrix@C=20pt@R=15pt{
    \alpha: \Hom_R(M,F_1 \otimes_R \omega) \ar[r] \ar[d]_{\rotatebox{90}{$\cong$}} 
    & \Hom_R(M,F_0 \otimes_R \omega) \ar[r] \ar[d]_{\rotatebox{90}{$\cong$}} 
    & \Hom_R(M,R/I \otimes_R \omega) \ar[d] \\
    \beta: \Hom_R(M, H_0(F_1\otimes_R^\L\dual)) \ar[r] 
    & \Hom_R(M, H_0(F_0\otimes_R^\L\dual)) \ar[r] 
    & \Hom_R(M, H_0(F\otimes_R^\L\dual)) \\
  }
\]
of complexes.
Since $\beta$ is exact, so is $\alpha$.
\end{proof}

Below is a direct corollary of Theorem \ref{nonCM} and is a non-Cohen--Macaulay version of \cite[Theorem 1.13]{DD}.
The result on the trace ideal improves \cite[Theorem A(i)]{A} when the ring is a complete local domain.
In particular, by \cite[Theorem 3]{I}, it refines the monomial conjecture.

\begin{cor}\label{nonCM original}
Let $R$ be a complete local ring, $I$ a proper ideal of $R$, $M$ a big Cohen--Macaulay $R$-module, and $\omega$ a canonical module of $R$.
Then $\tr_\omega (M)\nsubseteq I\omega$ provided one of the following holds:
\begin{enumerate}[\rm(1)]
\item There is $G$ satisfying the standard conditions on rank and height such that $H_0(G)=R/I$;
\item $I$ has finite projective dimension;
\item $I$ is a parameter ideal.
\end{enumerate}
Additionally, assume that $R$ is quasi-Gorenstein.
Then the trace ideal $\tr_R (M)$ is not contained in $I$.
In particular, the conductor of $R$ is not contained in $I$ if $R$ is an integral domain.
\end{cor}

\begin{proof}
All statements except the last one are consequences of Lemma \ref{fpd lemma} and Theorem \ref{nonCM}.
Suppose that $R$ is a complete quasi-Gorenstein local domain.
Let $R^{+}$ be the absolute integral closure, that is, the integral closure of $R$ inside an algebraic closure of its fraction field.
There is an $R^{+}$-algebra $B$ such that it is a big Cohen--Macaulay $R$-module; see \cite{And, And2, B, HH4, HH5, Mu, Sh} for instance. 
It is enough to show that $\tr_R (B)$ is contained in the conductor $\mathfrak{C}$ of $R$.
Let $z\in \tr_R (B)$ and $x/y\in \overline{R}$.
Since $R^{+}$ is integrally closed, so is $yR^{+}$.
The integral closure of $yR$ is equal to $yR^{+}\cap R$ by \cite[Proposition 1.6.1]{HS}.
We have $x\in yR^{+}\cap R\subseteq yB\cap R$.
By \cite[Theorem 3.12]{PR}, we get $zx\in yR$.
We conclude that $z\in (R:_K \overline{R})=\mathfrak{C}$.
\end{proof}

From \cite[Theorem 1.13]{DD}, \cite[Proposition 6.1]{S2} was recovered in the case of complete Gorenstein domains.
Similarly, the following corollary provides a new answer to Question \ref{ques 2}.

\begin{cor}\label{nonCM cor}
Let $R$ be a complete local domain of prime characteristic, $I$ a proper ideal of $R$, and $\omega$ a canonical module of $R$.
Then $\tau (\omega)\nsubseteq I\omega$ provided one of the following holds:
\begin{enumerate}[\rm(1)]
\item $R/I$ has finite phantom projective dimension;
\item $I$ has finite projective dimension;
\item $I$ is a parameter ideal.
\end{enumerate}
In particular, if $R$ is quasi-Gorenstein, the parameter test ideal of $R$ is not contained in $I$.
\end{cor}

\begin{proof}
Let $R^{+}$ be the absolute integral closure, that is, the integral closure of $R$ inside an algebraic closure of its fraction field.
Put $d=\dim R$.
It follows from \cite[Proposition 3.3(iii) and Theorem 5.1]{S1} that $0^\ast_{\H^d_\m(R)}=\ker (\H^d_\m(R)\to \H^d_\m(R)\otimes_R R^{+}:\eta\mapsto \eta\otimes 1)$.
We have an exact sequence
$$
\Hom_R(R^{+}, \Hom_R(\H^d_\m(R), E)) \xrightarrow{\phi} \Hom_R(\H^d_\m(R), E)\xrightarrow{\psi} \Hom_R(0^\ast_{\H^d_\m(R)}, E)\to 0, 
$$
where $E$ is the injective hull of the residue field of $R$ and for any $f\in\Hom_R(R^{+}, \Hom_R(\H^d_\m(R), E))$, $\phi(f)=f(1)$.
By definition, the equalities $\tau (\omega)=\ker\psi=\image\phi=\tr_\omega (R^{+})$ hold.
We see by \cite[Theorem 9.8]{HH2} that if $R/I$ has finite phantom projective dimension, then there is $G$ satisfying the standard conditions on rank and height such that $H_0(G)=R/I$.
It follows from \cite[Theorem 1.1]{HH4} and Corollary \ref{nonCM original} that $\tau (\omega)=\tr_\omega (R^{+})\nsubseteq I\omega$.
Suppose that $R$ is quasi-Gorenstein.
Since $R$ is complete, $\Hom_R(E/0_{E}^\ast, E)=\Ann_R (0^\ast_{E})$ holds.
By $\H^d_\m(R)= E$ and $\omega=R$, $\Ann_R (0^\ast_{E})$ is equal to $\tau (R)$ that is not contained in $I$, which implies that the parameter test ideal is not contained in $I$.
\end{proof}

\begin{ac}
The author would like to thank his host researcher Linquan Ma for valuable comments and helpful suggestions.
As stated in the main text, Ma proposed including Examples \ref{count. para. F-ideal CM} and \ref{count. para. s.o.p} in this paper, for which the author also wishes to reiterate his sincere gratitude.
The author is also grateful to Alberto Fernandez Boix and Naoyuki Matsuoka for their constructive advice.
The author was supported by JSPS Overseas Challenge Program for Young Researchers, and was partly supported by Grant-in-Aid for JSPS Fellows Grant Number 23KJ1117.
\end{ac}

\end{document}